\documentclass[10pt]{article}
\usepackage{amsmath,amssymb}
\usepackage{enumerate}
\usepackage[dvipsnames]{xcolor}

\allowdisplaybreaks[1]

\newtheorem{theorem}{Theorem}[section]

\newtheorem{lemma}[theorem]{Lemma}
\newtheorem{proposition}[theorem]{Proposition}

\newtheorem{definition}[theorem]{Definition}
\newtheorem{remark}[theorem]{Remark}

\newcommand{\Z}{\mathbb{Z}}

\renewcommand{\ker}{\operatorname{Ker}}
\newcommand{\id}{\operatorname{id}}

\newcommand{\Sym}{\operatorname{Sym}}
\newcommand{\aut}{\operatorname{Aut}}

\newcommand{\soc}{\operatorname{Soc}}

\newcommand{\Aut}{\operatorname{Aut}}

\newcommand{\Ret}{\operatorname{Ret}}
\newcommand{\gr}{\operatorname{gr}}

\newenvironment{proof}{\par\noindent{ Proof.}}{$\qed$\par\bigskip}
\newcommand{\qed}{\enspace\vrule  height6pt  width4pt  depth2pt}
\usepackage{color}

\begin{document}
\title{Simple solutions of the Yang--Baxter equation of cardinality $p^n$\thanks{The first author was partially
supported by the project PID2020-113047GB-I00/AEI/10.13039/501100011033 (Spain). }}
\author{F. Ced\'o \and J. Okni\'{n}ski
}
\date{}

\maketitle


\vspace{30pt} \noindent Keywords: Yang--Baxter equation,
set-theoretic solution,
 indecomposable solution,  simple solution, brace\\


\begin{abstract}
 For every prime number $p$ and integer $n>1$, a simple,
involutive, non-degenerate, set-theoretic solution $(X,r)$ of the
Yang--Baxter equation of cardinality $|X|=p^n$ is constructed.
Futhermore, for every non-(square-free) positive integer $m$ which
is not a square of a prime number, a non-simple, indecomposable,
irretractable, involutive, non-degenerate, set-theoretic solution
$(X,r)$ of the Yang--Baxter equation of cardinality $|X|=m$ is
constructed. A recent question of Castelli on the existence of singular solutions of certain type is also answered affirmatively.
\end{abstract}

\section{Introduction and preliminaries} \label{prelim}

An important open problem is to find all the solutions of the Yang--Baxter equation. Drinfeld in \cite{drinfeld} suggested to study the set-theoretic solutions of the Yang--Baxter equation. Gateva-Ivanova and Van den Bergh \cite{GIVdB} and Etingof, Schedler and Soloviev \cite{ESS} introduced the class of so called involutive non-degenerate set-theoretic solutions of the Yang--Baxter equation. The study of this important class of solutions has attracted a lot of attention during the last twenty years, see for example \cite{CedoSurvey} and the references therein.

Let $X$ be a non-empty set and  let  $r:X\times X \longrightarrow
X\times X$ be a map. For $x,y\in X$ we put $r(x,y) =(\sigma_x (y),
\gamma_y (x))$. Recall that $(X,r)$ is an involutive,
non-degenerate set-theoretic solution of the Yang--Baxter equation
if $r^2=\id$, all the maps $\sigma_x$ and $\gamma_y$ are bijective
maps from $X$ to itself and
  $$r_{12} r_{23} r_{12} =r_{23} r_{12} r_{23},$$
where $r_{12}=r\times \id_X$ and $r_{23}=\id_X\times \ r$ are maps
from $X^3$ to itself. Because $r^{2}=\id$, one easily verifies
that $\gamma_y(x)=\sigma^{-1}_{\sigma_x(y)}(x)$, for all $x,y\in
X$ (see for example \cite[Proposition~1.6]{ESS}).

\bigskip
\noindent {\bf Convention.} Throughout the paper a solution of the
YBE will mean an involutive, non-degenerate, set-theoretic
solution of the Yang--Baxter equation. \\

It is well known, see for example \cite[Proposition
8.2]{CedoSurvey}, that a map $r(x,y) = (\sigma_x (y),
\sigma^{-1}_{\sigma_x(y)}(x))$, defined for a collection of
bijections $\sigma_x, x\in X$, of the set $X$, is a solution of
the YBE if and only if
$\sigma_{x}\sigma_{\sigma_{x}^{-1}(y)}=\sigma_{y}\sigma_{\sigma_{y}^{-1}(x)}$
for all $x,y\in X$ and the maps $\gamma_y\colon X\rightarrow X$
defined by $\gamma_y(x)=\sigma^{-1}_{\sigma_x(y)}(x)$, for all
$x,y\in X$, are bijective. Furthermore, if $X$ is finite, then
$(X,r)$ is a solution of the YBE if and only if
$\sigma_{x}\sigma_{\sigma_{x}^{-1}(y)}=\sigma_{y}\sigma_{\sigma_{y}^{-1}(x)}$
    for all $x,y\in X$.

 To study solutions of the YBE, Rump \cite{R07} introduced a new algebraic structure called left brace. This allowed to construct many new families of solutions of the YBE \cite{BCJ,BCJO,BCV,CJOComm,GI18,Smok}.

A left brace is a set $B$ with two binary operations, $+$ and
$\circ$, such that $(B,+)$ is an abelian group (the additive group
of $B$), $(B,\circ)$ is a group (the multiplicative group of $B$),
and for every $a,b,c\in B$,
 \begin{eqnarray} \label{braceeq}
  a\circ (b+c)+a&=&a\circ b+a\circ c.
 \end{eqnarray}
In any left brace $B$  the neutral elements $0,1$ for the
operations $+$ and $\circ$ coincide. Moreover, there is an action
$\lambda\colon (B,\circ)\longrightarrow \aut(B,+)$, called the
lambda map of $B$, defined by $\lambda(a)=\lambda_a$ and
$\lambda_{a}(b)=-a+a\circ b$, for $a,b\in B$. We shall write
$a\circ b=ab$ and $a^{-1}$ will denote the inverse of $a$ for the
operation $\circ$, for all $a,b\in B$. A trivial brace is a left
brace $B$ such that $ab=a+b$, for all $a,b\in B$, i.e. all
$\lambda_a=\id$. The socle of a left brace $B$ is
$$\soc(B)=\{ a\in B\mid ab=a+b, \mbox{ for all
}b\in B \}.$$ Note that $\soc(B)=\ker(\lambda)$, and thus it is a
normal subgroup of the multiplicative group of $B$. The solution of
the YBE associated to a left brace $B$ is $(B,r_B)$, where
$r_B(a,b)=(\lambda_a(b),\lambda_{\lambda_a(b)}^{-1}(a))$, for all
$a,b\in B$ (see \cite[Lemma~2]{CJOComm}).

A left ideal of a left brace $B$ is a subgroup $L$ of the additive
group of $B$ such that $\lambda_a(b)\in L$, for all $b\in L$ and
all $a\in B$. An ideal of a left brace $B$ is a normal subgroup
$I$ of the multiplicative group of $B$ such that $\lambda_a(b)\in
I$, for all $b\in I$ and all $a\in B$. Note that
\begin{eqnarray}\label{addmult1}
ab^{-1}&=&a-\lambda_{ab^{-1}}(b)
\end{eqnarray}
 for all $a,b\in B$, and
    \begin{eqnarray} \label{addmult2}
     &&a-b=a+\lambda_{b}(b^{-1})= a\lambda_{a^{-1}}(\lambda_b(b^{-1}))= a\lambda_{a^{-1}b}(b^{-1}),
     \end{eqnarray}
for all $a,b\in B$. Hence, every left ideal $L$ of $B$ also is a
subgroup of the multiplicative group of $B$, and every  ideal $I$
of a left brace $B$ also is a subgroup of the additive group of
$B$. For example,  every Sylow subgroup of the additive group of a
finite left brace $B$ is a left ideal of $B$. Consequently, if a
Sylow subgroup of the multiplicative group $(B,\circ)$ is normal,
then it must be an ideal of the finite left brace $B$. It is also
known that $\soc(B)$ is an ideal of the left brace $B$ (see
\cite[Proposition~7]{R07}). Note that, for every ideal $I$ of $B$,
$B/I$ inherits a natural left brace structure. A nonzero left
brace $B$ is simple if $\{0\}$ and $B$ are the only ideals of $B$.

A homomorphism of left braces is a map $f\colon B_1\longrightarrow
B_2$, where $B_1,B_2$ are left braces, such that $f(a b)=f(a)
f(b)$ and $f(a+b)=f(a)+f(b)$, for all $a,b\in B_1$. Note that the
kernel $\ker(f)$ of a homomorphism of left braces $f\colon
B_1\longrightarrow B_2$ is an ideal of $B_1$.

Recall that if $(X,r)$ is a solution of the YBE,
with
$r(x,y)=(\sigma_x(y),\gamma_y(x))$, then
its structure group
$$G(X,r)=\gr(x\in X\mid xy=\sigma_x(y)\gamma_y(x)\mbox{
for all }x,y\in
X)$$
has a natural structure of a left brace such that
$\lambda_x(y)=\sigma_x(y)$, for all $x,y\in X$. The additive group
of $G(X,r)$ is the free abelian group with basis $X$.
The permutation group $\mathcal{G}(X,r)=\gr(\sigma_x\mid x\in X)$ of
$(X,r)$ is a subgroup of the symmetric group $\Sym_X$ on $X$.
The
map $x\mapsto \sigma_x$, from $X$ to $\mathcal{G}(X,r)$ extends to
a group homomorphism $\phi: G(X,r)\longrightarrow
\mathcal{G}(X,r)$.
and $\ker(\phi)=\soc(G(X,r))$. Hence there is a
unique structure of a left brace on $\mathcal{G}(X,r)$ such that
$\phi$ is a homomorphism of left braces; this is the natural
structure of a left brace on $\mathcal{G}(X,r)$. In particular,
\begin{equation*}
    \sigma_{\sigma_x(y)}=\phi(\sigma_x(y))=\phi(\lambda_x(y))=\lambda_{\phi(x)}(\phi(y))=\lambda_{\sigma_x}(\sigma_y)
\end{equation*}
for all $x,y\in X$.

Let $(X,r)$ and $(Y,s)$ be solutions of the YBE. We write
$r(x,y)=(\sigma_x(y),\gamma_y(x))$ and
$s(t,z)=(\sigma'_t(z),\gamma'_z(t))$, for all $x,y\in X$ and $t,z\in
Y$. A homomorphism of solutions $f\colon (X,r)\longrightarrow (Y,s)$
is a map $f\colon X\longrightarrow Y$ such that
$f(\sigma_x(y))=\sigma'_{f(x)}(f(y))$ and
$f(\gamma_y(x))=\gamma'_{f(y)}(f(x))$, for all $x,y\in X$. Since
$\gamma_y(x)=\sigma^{-1}_{\sigma_x(y)}(x)$ and
$\gamma'_z(t)=(\sigma')^{-1}_{\sigma'_t(z)}(t)$, it is clear that
$f$ is a homomorphism of solutions if and only if
$f(\sigma_x(y))=\sigma'_{f(x)}(f(y))$, for all $x,y\in X$.

In \cite{ESS}, Etingof, Schedler and Soloviev introduced the
retract relation on solutions  $(X,r)$ of the YBE. This is the
binary relation $\sim$ on $X$ defined by $x\sim y$ if and only if
$\sigma_x=\sigma_y$. Then, $\sim$ is an equivalence relation and
$r$ induces a solution $\overline{r}$ on the set
$\overline{X}=X/{\sim}$. The retract of the solution $(X,r)$ is
$\Ret(X,r)=(\overline{X},\overline{r})$. Note that the natural map
$f\colon X\longrightarrow \overline{X}:x\mapsto \bar x$ is an
epimorphism of solutions from $(X,r)$ onto $\Ret(X,r)$.

Recall that a solution $(X,r)$ is  said to be irretractable if
$\sigma_x\neq \sigma_y$ for all distinct elements $x,y\in X$, that
is $(X,r)=\Ret(X,r)$; otherwise the solution $(X,r)$ is
retractable.  Define $\Ret^1(X,r)=\Ret(X,r)$ and, for every
integer $n>1$, $\Ret^{n}(X,r)=\Ret(\Ret^{n-1}(X,r))$. A solution
$(X,r)$ is said to be a multipermutation solution if there exists
a positive integer $n$ such that $\Ret^n(X,r)$ has cardinality
$1$.

Let $(X,r)$ be a solution of the YBE. We say that $(X,r)$ is
indecomposable if $\mathcal{G}(X,r)$ acts transitively on $X$.

\begin{definition}
    A solution $(X,r)$ of the YBE is simple if $|X|>1$ and for every
    epimorphism $f:(X,r) \longrightarrow (Y,s)$ of solutions either $f$
    is an isomorphism or $|Y|=1$.
\end{definition}

In this context, the following result (Proposition~4.1 in
\cite{CO21} and Lemma~3.4 in \cite{CO22}) is crucial.

\begin{lemma}  \label{simple indec}
Assume that $(X,r)$ is a simple solution of the YBE. Then it is
indecomposable if $|X|>2$ and it is irretractable if $|X|$ is not
a prime number.
\end{lemma}

Finite simple solutions of the YBE were introduced by Vendramin in \cite{V}. His definition does not coincide with the above definition of simplicity, but for finite indecomposable solutions both definitions coincide by \cite[Proposition 2]{CCP}.
It is not known whether there exists a simple solution of the YBE of cardinality $m^2n$ for any integers $m,n>1$. This was claimed in \cite[Theorem 4.12]{CO21}, but the proof was  incorrect, though an example of a simple solution of non-square cardinality was given, see \cite{Corrigendum}. On the other hand, all the simple solutions of the YBE of non-prime cardinality constructed in \cite{CO21,CO22} have square cardinality.
In \cite{Cast} Castelli gives an example of a simple solution of the YBE of cardinality $27$, and two concrete simple solutions of the YBE of cardinality $12$. In \cite[Theorem 5.3]{CO24} it is proven that if
$n>1$ is an integer and $p$ is a prime divisor of $q-1$ for every
prime divisor $q$ of $n$, then there exists a simple solution of
the YBE of cardinality $p^2n$.
In this paper we construct, for every prime number $p$ and every positive integer $n$, a simple solution of the YBE of cardinality $p^n$, see Theorem~\ref{main}.
Moreover, for every non-(square-free) positive integer $m$ which
is not a square of a prime number, a non-simple, indecomposable and
irretractable solution
$(X,r)$ of the YBE of cardinality $|X|=m$ is
constructed in Section~\ref{mixed}. Finally, simple singular solutions (introduced by Rump in \cite{Rump23}) are constructed in Section~\ref{singular}. This answers a recent question of Castelli \cite{Castelli} affirmatively.

\section{Solutions of cardinality $m^2n$}\label{mixed}

By \cite[Theorem 4.5]{CO23} we know that if $(X,r)$ is a finite
indecomposable and irretractable solution, then $|X|$ is not
square-free. On the other hand Dietzel, Properzi and Trappeniers
have proven in \cite[page 17]{DPT} that if $|X|=p^2$ for a prime
$p$, then $(X,r)$ is indecomposable and irretractable if and only
if $(X,r)$ is simple. In this section we will show that if $|X|$
is not square-free and is not the square of a prime, then there
exists an indecomposable and irretractable solution $(X,r)$ of the
YBE which is not simple.

\begin{remark}
    {\rm Let $m>1$ be an integer which is not a prime number.
    Using the results of \cite{CO21}, we shall construct an indecomposable and
    irretractable solution of the YBE of cardinality $m^2$ which is not simple.
    Since $m$ is not prime, there exist integers $m_1,m_2>1$ such that $m=m_1m_2$.
    Let $X=(\Z/(m))^2$. Consider the family $(j_a)_{a\in\Z/(m)}$ of elements of
    $\Z/(m)$ such that $j_0=1$ and $j_a=m_1+1$, for all $a\in\Z/(m)\setminus\{ 0\}$.
    For every $(a_1,a_2)\in X$, let $\sigma_{(a_1,a_2)}\colon X\longrightarrow X$ be
    the map defined by
    $$\sigma_{(a_1,a_2)}(x,y)=(x+a_2, y-j_{x+a_2-a_1}),$$
    for all $(x,y)\in X$. Let $r\colon X\times X\longrightarrow X\times X$ be the map
    defined by
     $$r((a_1,a_2),(c_1,c_2))=(\sigma_{(a_1,a_2)}(c_1,c_2),\sigma^{-1}_{\sigma_{(a_1,a_2)}(c_1,c_2)}(a_1,a_2)),$$
     for all $(a_1,a_2)(c_1,c_2)\in X$. By \cite[Theorem 4.9]{CO21},
     $(X,r)$ is an indecomposable and irretractable solution of the YBE. Let $\pi\colon \Z/(m)\longrightarrow \Z/(m_1)$
     be the natural homomorphism. Consider the solution $(\Z/(m_1), s)$ of the YBE,
     where $$s(x,y)=(y-1,x+1),$$
     for all $x,y\in\Z/(m_1)$. Let $f\colon X\longrightarrow \Z/(m_1)$ be the map
     defined by $f(a_1,a_2)=\pi(a_2)$. Note that
     $$f(\sigma_{(a_1,a_2)}(c_1,c_2))=f(c_1+a_2,c_2-j_{c_1+a_2-a_1})=\pi(c_2-j_{c_1+a_2-a_1})=\pi(c_2)-1$$
     and $s(f(a_1,a_2),f(c_1c_2))=s(\pi(a_2),\pi(c_2))=(\pi(c_2)-1,\pi(a_2)+1)$.
     Hence $f$ is a homomorphism of solutions from $(X,r)$ to $(\Z/(m_1),s)$.
     Clearly $f$ is surjective. Hence the solution $(X,r)$ is not simple. In the proof of \cite[Proposition 6.2]{CO21}, one can see that $\mathcal{G}(X,r)$ is isomorphic to $B/\soc(B)$, where $B$ is some asymmetric product (\cite{CCS}) of the trivial braces $(\Z/(m))^m$ and $\Z/(m)$.}
\end{remark}

\begin{theorem}
    Let $m,n>1$ be integers.
    Let $X=\Z /(m)\times \Z/(n)\times \Z/(m)$. Consider $\sigma_{(a,b,c)}\in \Sym_X$ defined by
    $$\sigma_{(a,b,c)}(x,y,z)=(x-c,y-1,z-\delta_{(a,b),(x-c,y-1)})$$
    for all $a,c,x,z\in\Z/(m)$ and $b,y\in\Z/(n)$. Let $r\colon X\times X\longrightarrow X\times X$ be the map defined by
    $$r((a,b,c),(x,y,z))=(\sigma_{(a,b,c)}(x,y,z),\sigma^{-1}_{\sigma_{(a,b,c)}(x,y,z)}(a,b,c))$$
    for all $(a,b,c),(x,y,z)\in X$. Then $(X,r)$ is an indecomposable and irretractable solution of the YBE.
    Furthermore the map $f\colon X\longrightarrow \Z/(n)$ defined by $f(a,b,c)=b$ for all $(a,b,c)\in X$
    is an epimorphism of solutions from $(X,r)$ to $(\Z/(n),s)$, where $s(i,j)=(j-1,i+1)$ for all $i,j\in\Z/(n)$.
\end{theorem}

\begin{proof}
    Let $(a,b,c),(x,y,z)\in X$. Note that
    $$\sigma_{(a,b,c)}^{-1}(x,y,z)=(x+c,y+1,z+\delta_{(a,b),(x,y)}).$$
    To prove that $(X,r)$ is a solution of the YBE, since $X$ is finite, it is enough to check that
    \begin{equation}\label{cycle}
    \sigma^{-1}_{\sigma^{-1}_{(a,b,c)}(x,y,z)}\sigma^{-1}_{(a,b,c)}=
    \sigma^{-1}_{\sigma^{-1}_{(x,y,z)}(a,b,c)}\sigma^{-1}_{(x,y,z)}.\end{equation}
    For all $u,w\in \mathbb{Z}/(m)$ and $v\in \mathbb{Z}/(n)$ we have that
    \begin{align*}
    \sigma^{-1}_{\sigma^{-1}_{(a,b,c)}(x,y,z)}\sigma^{-1}_{(a,b,c)}(u,v,w)
    &=\sigma^{-1}_{(x+c,y+1,z+\delta_{(a,b),(x,y)})}(u+c,v+1,w+\delta_{(a,b),(u,v)})\\
    &=(u+c+z+\delta_{(a,b),(x,y)},v+2,w+\delta_{(a,b),(u,v)}+\delta_{(x+c,y+1),(u+c,v+1)})\\
    &=(u+c+z+\delta_{(a,b),(x,y)},v+2,w+\delta_{(a,b),(u,v)}+\delta_{(x,y),(u,v)})
    \end{align*}
    and
    \begin{align*}
        \sigma^{-1}_{\sigma^{-1}_{(x,y,z)}(a,b,c)}\sigma^{-1}_{(x,y,z)}(u,v,w)
        &=\sigma^{-1}_{(a+z,b+1,c+\delta_{(x,y),(a,b)})}(u+z,v+1,w+\delta_{(x,y),(u,v)})\\
        &=(u+z+c+\delta_{(x,y),(a,b)},v+2,w+\delta_{(x,y),(u,v)}+\delta_{(a+z,b+1),(u+z,v+1)})\\
        &=(u+c+z+\delta_{(a,b),(x,y)},v+2,w+\delta_{(a,b),(u,v)}+\delta_{(x,y),(u,v)}).
    \end{align*}
    Hence (\ref{cycle}) follows and therefore $(X,r)$ is a solution of the YBE.
     Note that
     $$\sigma^{-1}_{(1,i,0)}\sigma^{-1}_{(1,i-1,0)}\cdots \sigma^{-1}_{(1,1,0)}(0,0,0)=(0,i,0)$$
     for all $i\in\Z/(n)$. Hence $(0,i,0)$ is in the orbit of $(0,0,0)$
     by the action of $\mathcal{G}(X,r)$, for all $i\in\Z/(n)$. Note that
     $$\sigma^{-1}_{(1,i,1)}(0,i-1,0)=(1,i,0)$$
     for all $i\in \Z/(n)$. Hence $(1,i,0)$ is in the orbit of $(0,0,0)$ by the action of $\mathcal{G}(X,r)$, for all $i\in\Z/(n)$. Since
     $$\sigma^{-1}_{(j+1,i,1)}(j,i-1,0)=(j+1,i,0)$$
     for all $j\in \Z/(m)$ and $i\in \Z/(n)$, we have that
     $(j,i,0)$ is in the orbit of $(0,0,0)$ by the action of $\mathcal{G}(X,r)$, for all $j\in\Z/(m)$ and $i\in\Z/(n)$.
     We also have that
     $$\sigma^{-1}_{(j,i,0)}(j,i,k)=(j,i+1,k+1)$$
     for all $j,k\in\Z/(m)$ and $i\in \Z/(n)$. Hence $\mathcal{G}(X,r)$ acts transitively on $X$. Therefore $(X,r)$ is indecomposable.

     Let $(a,b,c),(a',b',c')\in X$ be elements such that $\sigma_{(a,b,c)}=\sigma_{(a',b',c')}$. Hence
     $$(x+c,y+1,z+\delta_{(a,b),(x,y)})=\sigma^{-1}_{(a,b,c)}(x,y,z)=\sigma^{-1}_{(a',b',c')}(x,y,z)=(x+c',y+1,z+\delta_{(a',b'),(x,y)})$$
     for all $(x,y,z)\in X$. Thus $c=c'$ and
     $$1=\delta_{(a,b),(a,b)}=\delta_{(a',b'),(a,b)},$$
     and therefore $(a,b,c)=(a',b',c')$. Hence $(X,r)$ is irretractable.

     Finally, it is easy to check that the map $f\colon X\longrightarrow \Z/(n)$, defined by $f(a,b,c)=b$,
     is an epimorphism of solutions from $(X,r)$ to $(\Z/(n), s)$. Therefore the result follows.
\end{proof}

 \section{Simple solutions of cardinality $p^{2n+1}$}\label{pn}

The example constructed in \cite[Theorem~4.12]{CO21} is not
correct. Hence, all the simple solutions of the YBE of
non-prime cardinality constructed in \cite{CO21, CO22} have square
cardinality. In \cite[Theorem 5.3]{CO24} it is proven that if
$n>1$ is an integer and $p$ is a prime divisor of $q-1$ for all
prime divisor $q$ of $n$, then there exists a simple solution of
the YBE of cardinality $p^2n$. In this section, for every prime
number $p$ and every integer $m>1$, we shall construct simple
solutions of the YBE of cardinality $p^m$.

\begin{remark}
    {\rm Let $m>1$ be an integer. Using the results from \cite{CO21},
    we shall construct a simple solution of the YBE of cardinality $m^2$.
        Let $X=(\Z/(m))^2$. Consider the family $(j_a)_{a\in\Z/(m)}$
        of elements of $\Z/(m)$ such that $j_0=1$ and $j_a=0$, for all
        $a\in\Z/(m)\setminus\{ 0\}$. For every $(a_1,a_2)\in X$, let
        $\sigma_{(a_1,a_2)}\colon X\longrightarrow X$ be the map defined by
        $$\sigma_{(a_1,a_2)}(x,y)=(x+a_2, y-j_{x+a_2-a_1}),$$
        for all $(x,y)\in X$. Let $r\colon X\times X\longrightarrow X\times X$
        be the map defined by
        $$r((a_1,a_2),(c_1,c_2))=(\sigma_{(a_1,a_2)}(c_1,c_2),\sigma^{-1}_{\sigma_{(a_1,a_2)}(c_1,c_2)}(a_1,a_2)),$$
        for all $(a_1,a_2)(c_1,c_2)\in X$.  By \cite[Theorem 4.9]{CO21},
        $(X,r)$ is a simple solution of the YBE.  In particular, for every prime number $p$ and positive integer
        $n$, there exists a simple solution of the YBE of cardinality $p^{2n}$. By \cite[Proposition 6.2]{CO21}, $\mathcal{G}(X,r)$ is isomorphic to some asymmetric product (\cite{CCS}) of the trivial braces $(\Z/(m))^m$ and $\Z/(m)$. In particular for a prime number $p$ and $m=p^n$, the permutation group $\mathcal{G}(X,r)$ is a $p$-group. }
\end{remark}

 Let $p$ be a prime number and let $n$ be a positive integer. Let $X=\Z/(p^n)\times\Z/(p)\times\Z/(p^n)$.
 Let $\pi\colon \Z/(p^n)\longrightarrow \Z/(p)$ be the canonical homomorphism. For $(a,b,c)\in X$, consider the map $\sigma_{(a,b,c)}\colon X\longrightarrow X$ defined by
 $$\sigma_{(a,b,c)}(x,y,z)=(x-c, y+\pi(a-x), z-\delta_{a,x-c}\delta_{b, y+\pi(a-x)}),$$
 for all $(x,y,z)\in X$. Note that $\sigma_{(a,b,c)}$ is bijective and
 $$\sigma^{-1}_{(a,b,c)}(x,y,z)=(x+c, y+\pi(x+c-a), z+\delta_{a,x}\delta_{b,y}),$$
 for all $(x,y,z)\in X$.

 \begin{theorem}\label{main} With the above notation, let $r\colon X^2\longrightarrow X^2$ be the map defined by
 $$r((a,b,c),(x,y,z))=(\sigma_{(a,b,c)}(x,y,z), \sigma^{-1}_{\sigma_{(a,b,c)}(x,y,z)}(a,b,c)),$$
 for all $(a,b,c),(x,y,z)\in X$. Then $(X,r)$ is a simple solution of the YBE.
\end{theorem}

\begin{proof}
To show that $(X,r)$ is a solution of the YBE it is enough to prove that
\begin{equation}\label{cycleset}
    \sigma^{-1}_{\sigma^{-1}_{(a,b,c)}(x,y,z)}\sigma^{-1}_{(a,b,c)}
    =\sigma^{-1}_{\sigma^{-1}_{(x,y,z)}(a,b,c)}\sigma^{-1}_{(x,y,z)},
    \end{equation}
    for all $(a,b,c),(x,y,z)\in X$.

Let $(a,b,c),(x,y,z),(u,v,w)\in X$. We have that
\begin{eqnarray*}
    \lefteqn{\sigma^{-1}_{\sigma^{-1}_{(a,b,c)}(x,y,z)}\sigma^{-1}_{(a,b,c)}(u,v,w)}\\
    &=&\sigma^{-1}_{(x+c, y+\pi(x+c-a), z+\delta_{a,x}\delta_{b,y})}(u+c, v+\pi(u+c-a), w+\delta_{a,u}\delta_{b,v})\\
    &=&(u+c+z+\delta_{a,x}\delta_{b,y}, v+\pi(u+c-a)+\pi(u+c+z+\delta_{a,x}\delta_{b,y}-x-c),\\
    &&\qquad w+\delta_{a,u}\delta_{b,v}+\delta_{x+c,u+c}\delta_{y+\pi(x+c-a),v+\pi(u+c-a)})\\
    &=&(u+c+z+\delta_{a,x}\delta_{b,y}, v+\pi(u-a-x+u+c+z+\delta_{a,x}\delta_{b,y}),
    w+\delta_{a,u}\delta_{b,v}+\delta_{x,u}\delta_{y+\pi(x),v+\pi(u)})
    \end{eqnarray*}
and
\begin{eqnarray*}
    \lefteqn{\sigma^{-1}_{\sigma^{-1}_{(x,y,z)}(a,b,c)}\sigma^{-1}_{(x,y,z)}(u,v,w)}\\
    &=&\sigma^{-1}_{(a+z, b+\pi(a+z-x), c+\delta_{x,a}\delta_{y,b})}(u+z, v+\pi(u+z-x), w+\delta_{x,u}\delta_{y,v})\\
    &=&(u+z+c+\delta_{x,a}\delta_{y,b}, v+\pi(u+z-x)+\pi(u+z+c+\delta_{x,a}\delta_{y,b}-a-z),\\
    &&\qquad w+\delta_{x,u}\delta_{y,v}+\delta_{a+z,u+z}\delta_{b+\pi(a+z-x),v+\pi(u+z-x)})\\
    &=&(u+z+c+\delta_{x,a}\delta_{y,b}, v+\pi(u-x-a+u+z+c+\delta_{x,a}\delta_{y,b}),
    w+\delta_{x,u}\delta_{y,v}+\delta_{a,u}\delta_{b+\pi(a),v+\pi(u)}).
\end{eqnarray*}
Note that $\delta_{x,u}\delta_{y,v}=\delta_{x,u}\delta_{y+\pi(x),v+\pi(u)}$
and $\delta_{a,u}\delta_{b+\pi(a),v+\pi(u)}=\delta_{a,u}\delta_{b,v}$. Hence
(\ref{cycleset}) holds and thus $(X,r)$ is a solution of the YBE.

We shall prove that $(X,r)$ is indecomposable. Note that
$$\sigma^{-1}_{(a,b,c)}(0,0,0)=(c, \pi(c-a), \delta_{a,0}\delta_{b,0}).$$
In particular, for $b=1$, we get that $(c,\pi(c-a),0)$ is in the
orbit of $(0,0,0)$ with respect to the action of
$\mathcal{G}(X,r)$ for all $a,c\in\Z/(p^n)$. Note that
$$\sigma^{-1}_{(c,\pi(c-a),0)}(c,\pi(c-a),i)=(c, \pi(c-a),i+1),$$
for all $a,c,i\in\Z/(p^n)$. Since $(c,\pi(c-a),0)$ is in the orbit of
$(0,0,0)$, we get that the orbit of $(0,0,0)$ is $X$. Hence
$(X,r)$ is indecomposable.

Let $f\colon (X,r)\longrightarrow (Y,s)$ be an epimorphism of
solutions of the YBE. Suppose that $f$ is not an isomorphism.
Since $(X,r)$ is indecomposable, by \cite[Lemma 3.3]{CO23},
$(Y,s)$ is indecomposable and $|f^{-1}(y)|=|f^{-1}(y')|$ for all
$y,y'\in Y$. Hence $|Y|=p^{k}$ for some $k\in\{ 0,1,\dots ,2n\}$.
First we shall prove that $k<2$.

Let $f(0,0,0)=y_0$ and
$$f^{-1}(y_0)=\{ (a_i,b_i,c_i)\mid 1\leq i\leq p^{2n+1-k}\}.$$
Assume that $(a_1,b_1,c_1)=(0,0,0)$. Note that
\begin{eqnarray*}
    \sigma_{(a_i,b_i,c_i)}\sigma^{-1}_{(0,0,0)}(a_j,b_j,c_j)&=&\sigma_{(a_i,b_i,c_i)}(a_j,b_j+\pi(a_j),c_j+\delta_{0,a_j}\delta_{0,b_j})\\
    &=&(a_j-c_i,b_j+\pi(a_i),c_j+\delta_{0,a_j}\delta_{0,b_j}-\delta_{a_i,a_j-c_i}\delta_{b_i,b_j+\pi(a_i)}),
\end{eqnarray*}
for all $i,j\in\{ 1,\dots ,p^{2n+1-k}\}$. Note that, for every
$1\leq i\leq p^{2n+1-k}$,
$$|\{ \sigma_{(a_i,b_i,c_i)}\sigma^{-1}_{(0,0,0)}(a_j,b_j,c_j)\mid 1\leq j\leq p^{2n+1-k}\}|=p^{2n+1-k}.$$
Hence, for every $1\leq i\leq p^{2n+1-k}$, since $f$ is a
homomorphism of solutions,
\begin{equation}\label{set1}
    f^{-1}(y_0)=
    \{ (a_j-c_i,b_j+\pi(a_i),c_j+\delta_{0,a_j}\delta_{0,b_j}-\delta_{a_i,a_j-c_i}\delta_{b_i,b_j+\pi(a_i)})\mid 1\leq j\leq p^{2n+1-k}\}.
    \end{equation}
We shall prove that there exists $i\in\{ 1,\dots ,p^{2n+1-k}\}$
such that $c_i$ is invertible in $\Z/(p^n)$. Suppose that $c_i\in
p(\Z/(p^n))$ for all $1\leq i\leq p^{2n+1-k}$. By (\ref{set1}),
for $j=1$ we get that
$(-c_i,\pi(a_i),\delta_{0,0}\delta_{0,0}-\delta_{a_i,-c_i}\delta_{b_i,\pi(a_i)})\in
f^{-1}(y_0)$. Thus there exists $1\leq l\leq p^{2n+1-k}$ such that
$$c_l=\delta_{0,0}\delta_{0,0}-\delta_{a_i,-c_i}\delta_{b_i,\pi(a_i)}=
1-\delta_{a_i,-c_i}\delta_{b_i,\pi(a_i)}.$$ Since $c_l\in
p(\Z/(p^n))$, we have that
$\delta_{a_i,-c_i}\delta_{b_i,\pi(a_i)}=1$. Hence  $a_i=-c_i$ and
$b_i=\pi(a_i)$, for all $1\leq i\leq p^{2n+1-k}$ and thus
$$(a_i,b_i,c_i)=(-c_i,\pi(-c_i),c_i)=(-c_i,0,c_i).$$
Hence
$$f^{-1}(y_0)=\{ (-c_i,0,c_i)\mid 1\leq i\leq p^{2n+1-k}\}.$$
In particular, $c_i=c_j$ if and only if $i=j$. Now we have,
$$y_0=f(-c_2,0,c_2)=f(\sigma_{(0,0,0)}\sigma^{-1}_{(-c_2,0,c_2)}(-c_2,0,c_2))=
f(\sigma_{(0,0,0)}(0,0,c_2+1))=f(0,0,c_2).$$ Hence $-c_2=0=c_1$,
and thus $c_2=c_1$, a contradiction, because $1<2\leq p^{2n+1-k}$.
Thus, indeed, there exists $i\in\{ 1,\dots ,p^{2n+1-k}\}$ such
that $c_i$ is invertible in $\Z/(p^n)$.

Since $c_i$ is invertible in $\Z/(p^n)$, its additive order is
$p^n$, and by (\ref{set1}),  it follows that
$$a_1-c_i,a_1-2c_i,\dots ,a_1-p^{n}c_i\in \{a_j\mid 1\leq j\leq p^{2n+1-k}\}\subseteq\Z/(p^n)$$
are $p^n$ distinct elements, and this implies that
\begin{equation}\label{groupai}\{ a_1,a_2,\dots ,a_{p^{2n+1-k}}\}=\Z/(p^n).
    \end{equation}

Suppose that $n=1$. In this case, $\pi=\id$ and $k\leq 2$. Suppose
that $k=2$. Thus $|f^{-1}(y_0)|=p$. Hence, there are $p$ different
first components of elements in
    (\ref{set1}). Moreover, for every $1\leq j\leq p$,
\begin{equation}\label{set2}
    f^{-1}(y_0)=\{ (a_j-c_i,b_j+a_i,c_j+\delta_{0,a_j}\delta_{0,b_j}-\delta_{a_i,a_j-c_i}\delta_{b_i,b_j+a_i})\mid 1\leq i\leq p\},
\end{equation}
(since there are $p$ different second components of elements
    in (\ref{set2})) and thus (by looking at the second
    components of elements in (\ref{set1}) and at the first components
    in (\ref{set2})), we get
$$\{ b_1,b_2,\dots ,b_p\}=\Z/(p)=\{ c_1,c_2,\dots,c_p\}.$$
In particular, for $j=1$, we have that
$$  f^{-1}(y_0)=\{ (-c_i,a_i,\delta_{0,0}\delta_{0,0}-\delta_{a_i,-c_i}\delta_{b_i,a_i})\mid 1\leq i\leq p\}.$$
Since
$\delta_{0,0}\delta_{0,0}-\delta_{a_i,-c_i}\delta_{b_i,a_i}\in\{
0,1\}$, it follows that $p=2$ and
$$f^{-1}(y_0)=\{ (0,0,0), (1,1,1)\}.$$
But, for $i=j=2$ in (\ref{set2}), we have that
$(1-1,1+1,1+\delta_{0,1}\delta_{0,1}-\delta_{1,0}\delta_{1,0})=(0,0,1)\in
f^{-1}(y_0)$, a contradiction. Hence $k< 2$, in this case.

Suppose that $n>1$. Let $a\in p(\Z/(p^n))$ be a nonzero element.
By (\ref{groupai}), there exists $j\in\{ 1,\dots ,p^{2n+1-k}\}$
such that $a_j=a$. We have
$$f(a_j,b_j,c_j+1)=f(\sigma^{-1}_{(a_j,b_j,0)}(a_j,b_j,c_j))=
f(\sigma^{-1}_{(a_j,b_j,0)}(0,0,0))=f(0,0,0)=y_0.$$  By induction
on $l$, one can see that
$$f(a_j,b_j,c_j+l)=y_0,$$
for all $1\leq l\leq p^n$. Hence, for every nonzero element $a\in
p(\Z/(p^n))$ there exists $j_a\in\{ 1,\dots ,p^{2n+1-k}\}$ such
that
    $$(a,b_{j_a},c)\in f^{-1}(y_0), $$
    for all $c\in\Z/(p^n)$.
In particular, $f(a,b_{j_a},0)=y_0$
and
$$\sigma_{(0,0,0)}\sigma^{-1}_{(a,b_{j_a},0)}(0,0,0)=\sigma_{(0,0,0)}(0,0,0)=
(0,0,-1)\in f^{-1}(y_0).$$ Note that
$\sigma_{(0,0,0)}\sigma^{-1}_{(a,b_{j_a},0)}(0,0,-1)=\sigma_{(0,0,0)}(0,0,-1)=
(0,0,-2)\in f^{-1}(y_0)$, and one can see that $(0,0,-l)\in
f^{-1}(y_0)$ by induction on $l$.  Hence,
$$(0,0,c)\in f^{-1}(y_0),$$
for all $c\in \Z/(p^n)$. Therefore,
$$\{ (a,b_{j_a},c)\mid a\in p(\Z/(p^n))\setminus\{ 0\},\; c\in \Z/(p^n)\}\cup\{(0,0,c)\mid c\in \Z/(p^n)\}\subseteq f^{-1}(y_0).$$
Thus $|f^{-1}(y_0)|\geq p^{n-1}p^n=p^{2n-1}$. By (\ref{groupai}),
there exists $(a_i,b_i,c_i)\in f^{-1}(y_0)$ such that
$a_i=1$. Thus
$$|f^{-1}(y_0)|>p^{2n-1}.$$
Hence, for $n>1$, we also have that $k<2$.

Therefore, either $|Y|=p$ or $|Y|=1$.

Suppose that $|Y|=p$. Since $(Y,s)$ is indecomposable, by
\cite[Theorem 2.13]{ESS} $(Y,s)$ is isomorphic to $(\Z/(p),s')$,
where $s'(x,y)=(y-1,x+1)$ for all $x,y\in\Z/(p)$. Let
$f(0,0,0)=y_0$. Let $g\colon (Y,s)\longrightarrow (\Z/(p),s')$ be
an isomorphism of solutions. We have that
$gf(\sigma_{(a,b,c)}(0,0,0))= g(y_0)-1$, for all $(a,b,c)\in X$.
But we know that
$$\sigma_{(0,1,0)}(0,0,0)=(0,0,0),$$
and thus $gf(\sigma_{(0,1,0)}(0,0,0))=gf(0,0,0)=g(y_0)$, a
contradiction. Hence $|Y|\neq p$.

Therefore $|Y|=1$ and $(X,r)$ is simple.
\end{proof}

\begin{proposition}
	Let $(X,r)$ be the simple solution descrived in Theorem \ref{main}. Then its permutation group $\mathcal{G}(X,r)$ is a $p$-group.
\end{proposition}

\begin{proof}
Consider the subgroup $I=\gr(\sigma_{(a,b,c)}\sigma^{-1}_{(0,0,0)}\mid (a,b,c)\in X)$ of the permutation group $\mathcal{G}(X,r)$ of the above simple solution $(X,r)$ of the YBE.  Note  that
$$\sigma_{(a,b,c)}\sigma^{-1}_{(x,y,z)}=\sigma_{(a,b,c)}\sigma^{-1}_{(0,0,0)}(\sigma_{(x,y,z)}\sigma^{-1}_{(0,0,0)})^{-1} \in I$$ for all $(a,b,c),(x,y,z)\in X$.
Furthermore,
$$\sigma^{-1}_{(a,b,c)}\sigma_{(x,y,z)}=\sigma_{\sigma^{-1}_{(a,b,c)}(x,y,z)}\sigma^{-1}_{\sigma^{-1}_{(x,y,z)}(a,b,c)}$$
also lies in $I$. Then clearly
$$\sigma_{(x,y,z)}^{-1}\sigma_{(a,b,c)}\sigma^{-1}_{(0,0,0)}\sigma_{(x,y,z)}\in I.$$
Hence $I$ is a normal subgroup of $\mathcal{G}(X,r)$. It is also clear that  $G(X, r)/I$
is cyclic generated by $\sigma_{(0,0,0)}I$.

Moreover, the elements of $I$ have the form
$$\sigma_{(a_1,b_1,c_1)}\sigma^{-1}_{(0,0,0)}\sigma_{(a_2,b_2,c_2)}\sigma^{-1}_{(0,0,0)}\cdots \sigma_{(a_k,b_k,c_k)}\sigma^{-1}_{(0,0,0)},$$
for a positive integer $k$ and $(a_1,b_1,c_1),\dots ,(a_k,b_k,c_k)\in X$.
Since
$$\sigma_{(a,b,c)}\sigma^{-1}_{(0,0,0)}(x,y,z)=\sigma_{(a,b,c)}(x,y+\pi(x),z+\delta_{0,x}\delta_{0,y})=(x-c,y+\pi(a),z+\delta_{0,x}\delta_{0,y}-\delta_{a,x-c}\delta_{b,y+\pi(a)}),$$
we have that
\begin{eqnarray*}
	\lefteqn{\sigma_{(a_1,b_1,c_1)}\sigma^{-1}_{(0,0,0)}\cdots \sigma_{(a_k,b_k,c_k)}\sigma^{-1}_{(0,0,0)}(x,y,z)}\\
	&=(x-c_k-\dots -c_1,y+\pi(a_k+\dots +a_1),
	z+\delta_{0,x}\delta_{0,y}-\delta_{a_k,x-c_k}\delta_{b_k,y+\pi(a_k)}\\
	&\qquad +\dots
	+\delta_{0,x-c_k-\dots -c_2}\delta_{0,y+\pi(a_k+\dots +a_2)}-\delta_{a_1,x-c_k-\dots -c_1}\delta_{b_1,y+\pi(a_k+\dots +a_1)}).
\end{eqnarray*}
Hence
\begin{eqnarray*}
	\lefteqn{(\sigma_{(a_1,b_1,c_1)}\sigma^{-1}_{(0,0,0)}\cdots \sigma_{(a_k,b_k,c_k)}\sigma^{-1}_{(0,0,0)})^{p^n}(x,y,z)}\\
	&=(x,y,	z+g(a_1,\dots, a_k,b_1,\dots ,b_k,c_1,\dots ,c_k,x,y)),
\end{eqnarray*}
for some $g(a_1,\dots, a_k,b_1,\dots ,b_k,c_1,\dots ,c_k,x,y)\in \Z/(p^n)$. Note that for every positive integer $t$,
\begin{eqnarray*}
	\lefteqn{(\sigma_{(a_1,b_1,c_1)}\sigma^{-1}_{(0,0,0)}\cdots \sigma_{(a_k,b_k,c_k)}\sigma^{-1}_{(0,0,0)})^{tp^n}(x,y,z)}\\
	&=(x,y,	z+tg(a_1,\dots, a_k,b_1,\dots ,b_k,c_1,\dots ,c_k,x,y)).
\end{eqnarray*}
Hence
$$	(\sigma_{(a_1,b_1,c_1)}\sigma^{-1}_{(0,0,0)}\cdots \sigma_{(a_k,b_k,c_k)}\sigma^{-1}_{(0,0,0)})^{p^{2n}}=\id.$$
Therefore $I$ is a $p$-subgroup of $\mathcal{G}(X,r)$. Note that
$$\sigma^{-1}_{(0,0,0)}(x,y,z)=(x,y+\pi(x),z+\delta_{0,x}\delta_{0,y}).$$
Hence
$$(\sigma^{-1}_{(0,0,0)})^{p}(x,y,z)=(x,y,z+\delta_{0,x}\delta_{0,y}+\dots +\delta_{0,x}\delta_{0,y+\pi((p-1)x)}),$$
and thus
$$(\sigma^{-1}_{(0,0,0)})^{p^{n+1}}(x,y,z)=(x,y,z+p^n(\delta_{0,x}\delta_{0,y}+\dots +\delta_{0,x}\delta_{0,y+\pi((p-1)x)})).$$
Therefore
$$(\sigma^{-1}_{(0,0,0)})^{p^{n+1}}=\id,$$
and this implies that $\mathcal{G}(X,r)$ is a $p$-group.
\end{proof}

\section{Singular simple solutions of cardinality $p^{2n}$}\label{singular}

In \cite{Rump23} Rump introduced singular left braces.
\begin{definition}
	Let $B$ be a finite left brace. A prime divisor $p$ of $|B|$ is called singular if there exists a finite indecomposable solution $(X,r)$ of the YBE such that $B\cong\mathcal{G}(X,r)$ and $p$ is not a divisor of $|X|$. We call $B$ singular if $B$ admits a singular prime. In this case, we say that the solution $(X,r)$ is singular.
\end{definition}

In \cite{Rump23} Rump proved that among the left braces of cardinality less than $36$, there is a unique singular left brace, up to isomorphism, it has order $24$ and it is the permutation group of an indecomposable and irretractable solution of cardinality $8$. In \cite[Example 21]{Castelli}, it is proved that this singular solution has finite primitive level (see \cite{CO21}), in particular it is not simple. Then Castelli in \cite{Castelli} constructs an infinite family of singular, indecomposable and irretractable solutions of the YBE and all of them have finite primitive level. He asks whether there exist singular solutions which have no finite primitive level. We answer this question in affirmative in the following result.

\begin{theorem}
	Let $p$ be an odd prime number which is not a Fermat prime number. Let $n$ be a positive integer and let $q$ be an odd prime divisor of $p-1$. Then there exists a singular simple solution $(X,r)$ of cardinality $p^{2n}$. Furthermore, $|\mathcal{G}(X,r)|=p^mq$ for some positive integer $m$.
\end{theorem}

\begin{proof} Note that since $p$ is an odd prime and it is not a Fermat prime number, there exists an odd prime divisor $q$ of $p-1$.
	Since $q$ is a divisor of $p-1$, there exists $t\in\Aut(\Z/(p^n))$ of order $q$. Let $T=\gr(t)$. Note that the orbit of every nonzero element $a\in\Z/(p^n)$ by the action of $T$ has cardinality $q$. Let $a\in\Z/(p^n)$ be a nonzero element. Since $p$ is an odd prime, $a\neq -a$. Suppose that $-a$ is in the orbit of $a$ by the action of $T$. Since the orbit of $a$ has cardinality $q$, there exists an integer $1<k<q$ such that $t^k(a)=-a$. Then $t^{2k}(a)=t^k(-a)=-t^k(a)=a$. Hence $q$ is a divisor of $2k$.  Since $q$ is an odd prime, $q$ is a divisor of $k$ and we get a contradiction because $1<k<q$. Thus for every nonzero $a\in\Z/(p^n)$, $-a$ is not in the orbit of $a$ by the action of $T$. Let $I$ be the set of the orbits in $\Z/(p^n)$ by the action of $T$. We can choose an element $a_i$ of each orbit $i\in I$ in such a way that $-a_i=a_l$ for some $l\in I$. We define a family $(j_a)_{a\in\Z/(p^n)}$ of elements of $\Z/(p^n)$ as follows: $j_0=1$  and
	$$j_{t^k(a_i)}=t^k(-1)+1,$$
	for all $i\in I\setminus\{\{ 0 \}\}$ and all $k\in \Z$. In particular, $j_{a_{i}}=0$ for every $i\in I\setminus\{\{ 0 \}\}$.
Let $a\in \Z/(p^n)\setminus\{ 0\}$. There exist $i,l\in I$ and $k\in\Z$ such that $a=t^k(a_i)$ and $-a_i=a_l$. Now we have that
	$$j_{-a}=j_{-t^k(a_i)}=j_{t^k(-a_i)}=j_{t^k(a_l)}=t^k(-1)+1=j_{t^k(a_i)}=j_{a},$$
	and for every $s\in\Z$, we also have that
	$$j_{t^s(a)}-j_0=j_{t^{s+k}(a_i)}-j_0=t^{s+k}(-1)=t^s(j_{t^k(a_i)}-j_0)=t^s(j_{a}-j_0).$$
	Let $X=\Z/(p^n)\times \Z/(p^n)$ and let $r\colon X\times X\longrightarrow X\times X$ be the map defined by
	$$r((a_1,a_2),(c_1,c_2))=(\sigma_{(a_1,a_2)}(c_1,c_2),\sigma^{-1}_{\sigma_{(a_1,a_2)}(c_1,c_2)}(a_1,a_2)),$$
	where
	$$\sigma_{(a_1,a_2)}(c_1,c_2)=(t(c_1)+a_2, t(c_2-j_{t(c_1)+a_2-a_1}))$$
	for all $a_1,a_2,c_1,c_2\in\Z/(p^n)$. By \cite[Theorem 3.1]{CO24}, $(X,r)$ is a solution of the YBE.
	Let $a\in\Z/(p^n)$ be a nonzero element. Let $V_{a,1}=\gr(j_c-j_{c+t^z(a)}\mid c\in\Z/(p^n),\; z\in\Z)$. For every $m>1$, define $V_{a,m}=V_{a,m-1}+\gr(j_c-j_{c+v}\mid c\in\Z/(p^n),\; v\in V_{a,m-1})$. Then $V_a=\sum_{m=1}^{\infty}V_{a,m}$ is a subgroup of $\Z/(p^n)$.
	There exist $i\in I$ and $k\in\Z$ such that $a=t^k(a_i)$. Note that $1=1-0=j_0-j_{a_i}\in V_{a,1}\subseteq V_a$ because $a_i$ is in the $T$-orbit of $a$. Hence $V_a=\Z/(p^n)$, for all nonzero $a\in\Z/(p^n)$. By \cite[Theorem 3.5]{CO24}, the solution $(X,r)$ is simple.
	Clearly $|X|=p^{2n}$. Note that for every positive integer $m$
	$$\sigma_{(0,0)}^m(u,v)=(t^m(u),w),$$
	for some $w\in \Z/(p^n)$. Let $k$ be the order of $\sigma_{(0,0)}$. We have that
	$$(u,v)=\sigma^{k}_{(0,0)}(u,v)=(t^k(u),v).$$
	Hence $t^k(u)=u$, for all $u\in\Z/(p^n)$. Since $t$ has order $q$, we get that $q$ is a divisor of $k$. Therefore $q$ is a divisor of $|\mathcal{G}(X,r)|$. Since $q\neq p$, the solution $(X,r)$ is singular and $\mathcal{G}(X,r)$ is a singular left brace.
	
	Note that the automorphism $t$ is defined by the multiplication by the invertible element $t(1)\in\Z/(p^n)$ of multiplicative order $q$. Since the elements of the form $1+pz\in\Z/(p^n)$ have multiplicative order a power of $p$, we have that $t(1)-1\notin p(\Z/(p^n))$. Hence $t(1)-1$ is invertible in $\Z/(p^n)$. Therefore $t-\id\in\Aut(\Z/(p^n))$. By \cite[Section 4]{CO24}, this allows us to describe the left brace structure of $\mathcal{G}(X,r)$.  Namely, by \cite[Proposition 4.1, the comments before Lemma 4.2 and Theorem 4.5]{CO24}, we have that $\mathcal{G}(X,r)$ is the asymmetric product $\overline{H}\rtimes_{\circ}A_1$ of a trivial brace $\overline{H}$ of order a power of $p$ and the semidirect product $A_1=\Z/(p^n)\rtimes T$ of the trivial braces $\Z/(p^n)$ and $T$. In particular, $|\mathcal{G}(X,r)|=p^mq$, for some positive integer $m$.
\end{proof}

\vspace{30pt}
 \noindent \begin{tabular}{llllllll}
  F. Ced\'o && J. Okni\'{n}ski \\
 Departament de Matem\`atiques &&  Institute of
Mathematics \\
 Universitat Aut\`onoma de Barcelona &&   University of Warsaw\\
08193 Bellaterra (Barcelona), Spain    &&  Banacha 2, 02-097 Warsaw, Poland \\
 Ferran.Cedo@uab.cat && okninski@mimuw.edu.pl\\

\end{tabular}

\end{document}